\documentclass[12pt]{amsart}

\usepackage{amsfonts,amsmath,amsthm,amssymb,graphicx,tikz}
\usepackage{amsmath}

\tikzstyle{vertex}=[circle, draw, inner sep=0pt, minimum size=6pt]

\theoremstyle{plain}
\newtheorem{thm}{Theorem}[section]

\newtheorem{lem}[thm]{Lemma}

\theoremstyle{definition}

\newtheorem{exmp}[thm]{Example}

\theoremstyle{remark}
\newtheorem{rem}[thm]{Remark}

\newcommand{\g}{\Gamma}

\begin{document}


\title{Correlation Clustering for General Graphs}

\author[L. Parsaei-Majd]{Leila Parsaei-Majd}

\address{L. Parsaei-Majd, Hasso Plattner Institute, University of Potsdam, Germany}

\email{leila.parsaei84@yahoo.com}




\begin{abstract}
Correlation clustering provides a method for separating the vertices of a signed graph into the optimum number of clusters without specifying that number in advance. The main goal in this type of clustering is to minimize the number of disagreements: the number of negative edges inside clusters plus the number of positive edges between clusters. 
In this paper, we present an algorithm for correlation clustering in general case. Also, we show that there is a necessary and sufficient condition under which the lower bound, maximum number of edge disjoint weakly negative cycles, is equal to minimum number of disagreements. Finally, we prove that the presented algorithm gives a $2$-approximation for a subclass of signed graphs.

\end{abstract}






\maketitle


\section{Introduction}
We can say clustering is the most important tool for partitioning data points into clusters based on their similarity. \textit{Correlation clustering} provides a method for separating the vertices of a signed graph into the optimum number of clusters without specifying that number in advance. 
Let $G$ be a graph with vertex set $V$ ($|V|=n$) and edge set $E$.
All graphs considered in this paper are undirected, finite, and simple (without loops or multiple edges). A {\em signed graph} is a graph in which every edge has a sign $+$ or $-$. In fact, a signed graph $\g$ is a pair $(G, \sigma)$, where $G=(V, E)$ is a graph, called the underlying graph, and $\sigma : E \rightarrow \{-1, +1\}$ is the sign function
or signature. Often, we write $\g=(G, \sigma)$ to mean that the underlying graph is $G$. Let $H$ be a signed graph. We say $H$ is a \textit{forbidden subgraph} of $\g$ if $\g$ does not contain $H$ as a subgraph. The main goal in this type
of clustering is separation of vertices of a signed graph into clusters in order to minimize the number of disagreements: the number of negative edges inside clusters and
the number of positive edges between clusters. Of course, we can interpret a positive edge and a negative edge as the similarity and dissimilarity of the corresponding vertices, respectively, and we want to cluster similar vertices into the same groups and dissimilar vertices into different groups. Correlation clustering was introduced by Bansal, Blum, and Chawla in \cite{cc}, and they obtained a $3$-approximation algorithm for complete graphs. Also, authors in \cite{cc15} gained a $3$-approximation algorithm for complete $k$-partite graphs. Very recently, Cohen-Addad et al. provided $a \sim 2 - 2/13 < 1.847$-approximation in sub-linear ($\tilde{\mathcal{O}}(n)$) sequential time or in sub-linear ($\tilde{\mathcal{O}}(n)$) space in the streaming setting using an efficient combinatorial algorithm \cite{cc24}. \\
Now, we are going to recall the results for general graphs (not necessarily complete graphs). In 2006, Demaine et al. in \cite{cc06} gave an $\mathcal{O}(\mathrm{log}n)$-approximation algorithm for general graphs based on a linear-programming rounding and the region-growing technique introduced in the seminal paper of Leighton and Rao \cite{[17]} on multicommodity max-flow min-cut theorems. The authors have shown that the problem for general graphs is equivalent to minimum multicut problem, so it is APX-hard and difficult to approximate better than $\Theta(\mathrm{log}n)$. Moreover, they provided an $\mathcal{O}(r^3)$-approximation algorithm for $K_{r,r}$-minor-free graphs which uses a rounding technique introduced by Tardos and Vazirani \cite{[27]} in a paper on max-flow min-multicut ratio and based on a lemma of Klein et al. \cite{[16]}. In the following, we recall a new result from \cite{ccbound}. 
Consider $V$ as a vertex set of size $n$ and $E \subseteq V \times V$ as the edge set. There are two nonnegative, real-valued weights $w^{+}_{uv}, w^{-}_{uv}$ for every (unordered) edge $(u, v) \in E$. Any “positive” $w^{+}_{uv}$ (resp. “negative” $w^{-}_{uv}$) weight expresses the benefit of clustering $u$ and $v$ together (resp. separately). This input can equivalently be represented as a graph $G$ with vertex set $V$ and edge set $E$ with edge weights $w^{+}_{uv}, w^{-}_{uv}$, for $(u, v) \in E$. The goal of correlation clustering is to find a clustering (i.e., an injective function expressing clustermembership) $\mathcal{C} : V \rightarrow \Bbb{N}^{+}$ that minimizes 
$$\sum_{(u,v) \in E,\mathcal{C}(u)=\mathcal{C}(v)} w^{-}_{uv} + \sum_{(u,v) \in E,\mathcal{C}(u)\neq \mathcal{C}(v)} w^{+}_{uv}.$$
\textbf{Global Weight Bound (GWB)} condition is the following equality:
$${\binom n2}^{-1} \sum_{(u,v) \in E} w^{+}_{uv} + {\binom n2}^{-1} \sum_{(u,v) \in E} w^{-}_{uv} \geqslant \max_{(u,v)\in E}\mid w^{+}_{uv} - w^{-}_{uv} \mid.$$
In 2021, Mandaglio et al. proved for general graphs satisfying Global Weight Bound (GWB) condition, there is a $5$-approximation algorithm for correlation clustering. In fact, they consider a wide range of edge weights for general graphs, and they presented an interesting result that leads to a $5$-approximate algorithm.
According to above notes, we can consider a signed graph with the vertex set $V$ and the edge set $E$, as a graph with edge weight $( w^{+}_{uv},  w^{-}_{uv}) \in \{(1, 0), (0, 1)\}$ for $(u, v) \in E$. 
However, signed graphs (not signed complete graphs) do not satisfy the GWB condition since the number of edges in a graph with $n$ vertices is ${\binom n2}$ if and only if the graph is complete. Therefore, if the underlying graph is not complete, the GWB condition cannot hold for the signed graphs. \\
A typical application of correlation clustering is clustering web pages based on similarity scores. Moreover, correlation clustering is applicable to a multitude of other problems in different domains, including duplicate detection and similarity joins, spam detection, co-reference resolution, biology, image segmentation, social networks, clustering aggregation, and
experiment design. We refer readers to \cite{ccbook22} for more details on the mentioned applications of correlation clustering. 

In this paper, we focus on signed general graphs and present an algorithm (C.C Algorithm) for correlation clustering. We know that maximum number of edge disjoint weakly negative cycles (with exactly one negative edge) is a lower bound for minimum number of disagreements. We provide a necessary and sufficient condition that this lower bound is equal to the minimum number of disagreements. Finally, we prove that C.C Algorithm gives a $2$-approximation for a subclass of signed graphs forbidding certain structures. Consider a signed graph $\g=(G, \sigma)$. $E^{+}$ and $E^{-}$ are the sets of positive and negative edges. The positive and negative subgraphs,
$$\g ^{+}=(V, E^{+}) \quad \text{and}\quad \g ^{-}=(V, E^{-}),$$
are unsigned graphs. The degree of a vertex has several generalizations to signed graphs. The
\textit{underlying degree} $d_{G}(v)$ is the degree in the underlying undirected graph, that is, the total number of
edges incident with $v$. The \textit{positive} or \textit{negative degree}, $d_{\g}^{+}(v)$ or $d_{\g}^{-}(v)$ is the degree in the positive subgraph ${\g}^{+}$ or the negative subgraph ${\g}^{-}$, respectively. 
Moreover, we can define the positive and negative degree of two subsets $V_1$ and $V_2$ of $V$, and denote them by $D^{+}_{\g}(V_1, V_2)$ and $D^{-}_{\g}(V_1, V_2)$, respectively. 
When $\g$ is clear from the context, for simplicity we can write $D^{+}(V_1, V_2)$ and $D^{-}(V_1, V_2)$. 
For a subset $S$ of $V$, the notation $[S]$ means an induced subgraph of $\g$ with the vertex set $S$ and the edge set containing all edges among vertices in $S$. An induced cycle is a cycle with no chord.  
The sign of a walk $W = e_1 e_2 \cdots e_{\ell}$ is the product of its edge signs: 
$$\sigma(W) := \sigma(e_1)\sigma(e_2) \cdots \sigma(e_{\ell}).$$
Thus, a walk is either positive or negative, depending on whether it has an even or
odd number of negative edges. A cycle is a closed walk, so cycles in a signed graph
based on the product of their edge signs are considered either positive or negative. In this paper, we call a cycle 'strongly positive' or 'weakly negative' if all of the edges are positive or it contains exactly one negative edge, respectively. We say two cycles are \textit{adjacent} if they have at least one common edge.

The concept of correlation clustering for a signed graph $\g$ is to partition $V$ so as to either minimize the sum of intra-cluster negative edges plus the sum of inter-cluster positive edges (min-disagreement), or maximize the sum of intra-cluster
positive edges plus the sum of inter-cluster negative edges (max-agreement).
The two formulations are equivalent in terms of exact optimization
and complexity class (both are \textbf{NP-hard}).

For a positive integer $k$, a signed graph $\g$ is called $k$-clusterable if $V = V_1 \cup \cdots \cup V_k$ such that the sets $V_1, \ldots, V_k$ are pairwise disjoint and all positive edges are inside the clusters $V_i$ for $i=1, \ldots, k$, and all negative edges are between clusters. 

\begin{thm}[{\cite{Davis}}]\label{thm:davis}
	A signed graph $\g$ is $k$-clusterable for some positive integer $k$, if and only if no cycle has exactly one negative edge.
\end{thm}

In this paper, we show positive edges in blue and negative edges in red in the figures.

\section{An algorithm for correlation clustering}
In this section, we are going to present an algorithm to separate the vertices of a signed graph into clusters.\\

\textbf{C.C Algorithm:} Let $\g$ be a signed graph. We can cluster the vertices of $\g$ as follows:
\begin{itemize}
	\item[(1).] 
	First put isolated vertices and vertices that all edges incident with are negative in separate and new clusters. That is, each of such vertices belongs to a new cluster. Then, remove these vertices from $\g$ in the next step. 
	
	\item[(2).] 
	Put the vertices of each strongly positive triangle in a new cluster and delete its vertices of the updated $\g$ obtained from step (1). If there are some strongly positive triangles which are connected to each other by a common edge or a vertex, we put all of their vertices in a new cluster, and remove their vertices from $\g$ obtained from the previous process.
	
	\item[(3).] 
	Find a maximum matching among the remaining positive edges of the new $\g$ obtained from step (2), and put the vertices of each element of the maximum matching in a new cluster. That is, each of the new clusters contains two vertices. Then, remove the vertices of the maximum matching from the updated $\g$.
	
	\item[(4).] 
	Put each of the remaining vertices of the new $\g$ obtained from step (3) in a new cluster, and add the remaining edges of the initial signed graph $\g$ between clusters or inside them so that all of the vertices and edges of $\g$ include in this clustering. 
	
	\item[(5).] 
	If there are two clusters like $V_1$ and $V_2$, with $D^{-}_{\g}(V_1, V_2) = 0$ and $D^{+}_{\g}(V_1, V_2) >= 1$, merge two clusters with the largest positive degree. Then calculate the new degrees based on the new clusters. Note that in this step we ignore the clusters created in step (1). That is, we do not compute the above values for the clusters created in step (1). 
	
	\item[(6).]
	Repeat step (5) until there is no pair of clusters with zero negative degree and non-zero positive degree.	
\end{itemize}

\begin{exmp}
	As a simple example, consider the signed graph $\g$ given in Fig. \ref{Ex2}. Figures given in Figs. \ref{ccalgo1}, \ref{ccalgo2}, \ref{ccalgo4} show the steps of C.C Algorithm to separate the vertices of $\g$ into four clusters. To further explain C.C Algorithm in details, we apply it step by step to $\g$. In $\g$ there is no isolated vertex but all edges incident with $v_8$ are negative, so we put $v_8$ in a new cluster $V_1$. Then we remove $v_8$ from $\g$. $v_1 v_3 v_7$ is the only strongly positive triangle in $\g$, so by step (2) we put $v_1, v_3$ and $v_7$ in a cluster $V_2$, see part (a) in Fig. \ref{ccalgo1}. Remove vertices $v_1, v_3$ and $v_7$ of $\g$, and the new $\g$ obtained from step (1) and step (2) is given in part (a$'$) in Fig. \ref{ccalgo1}. Now, we should find a maximum matching among the positive edges of the updated $\g$. The maximum matching contains just the positive edge $v_4v_5$, and by step (3), we put vertices $v_4$ and $v_5$ in a new cluster $V_3$, see part (b) in Fig. \ref{ccalgo2}. Remove vertices $v_4$ and $v_5$ from the updated $\g$ in part (a$'$) in Fig. \ref{ccalgo1}. By step (4), put each of the remaining vertices (of 
\begin{figure}[!htb]
	\minipage{0.56\textwidth}
	\includegraphics[width=\linewidth]{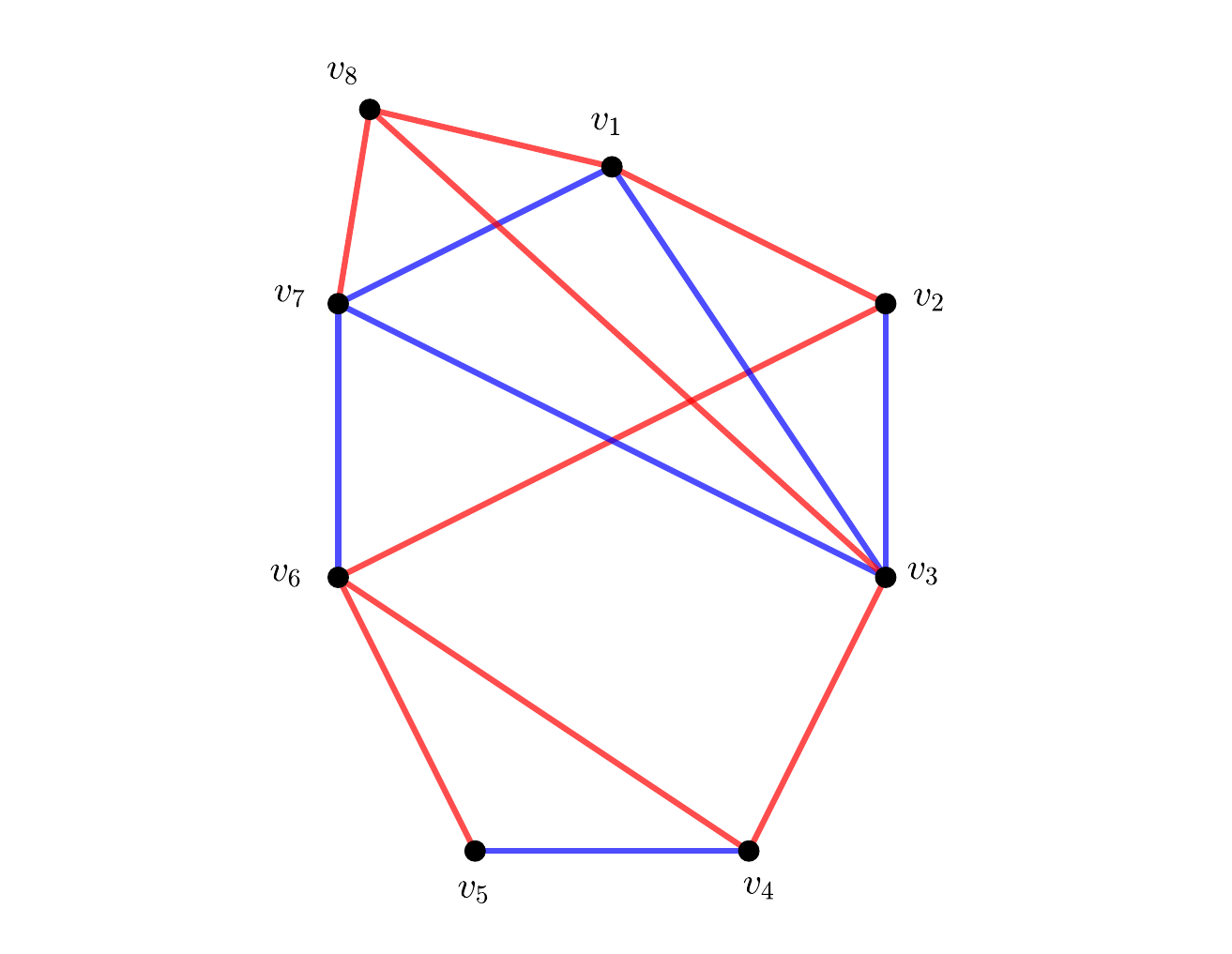}
	\caption{Signed graph $\g$.}\label{Ex2}
	\endminipage
\end{figure}
\begin{figure}[!htb]
	\minipage{0.62\textwidth}
	\includegraphics[width=\linewidth]{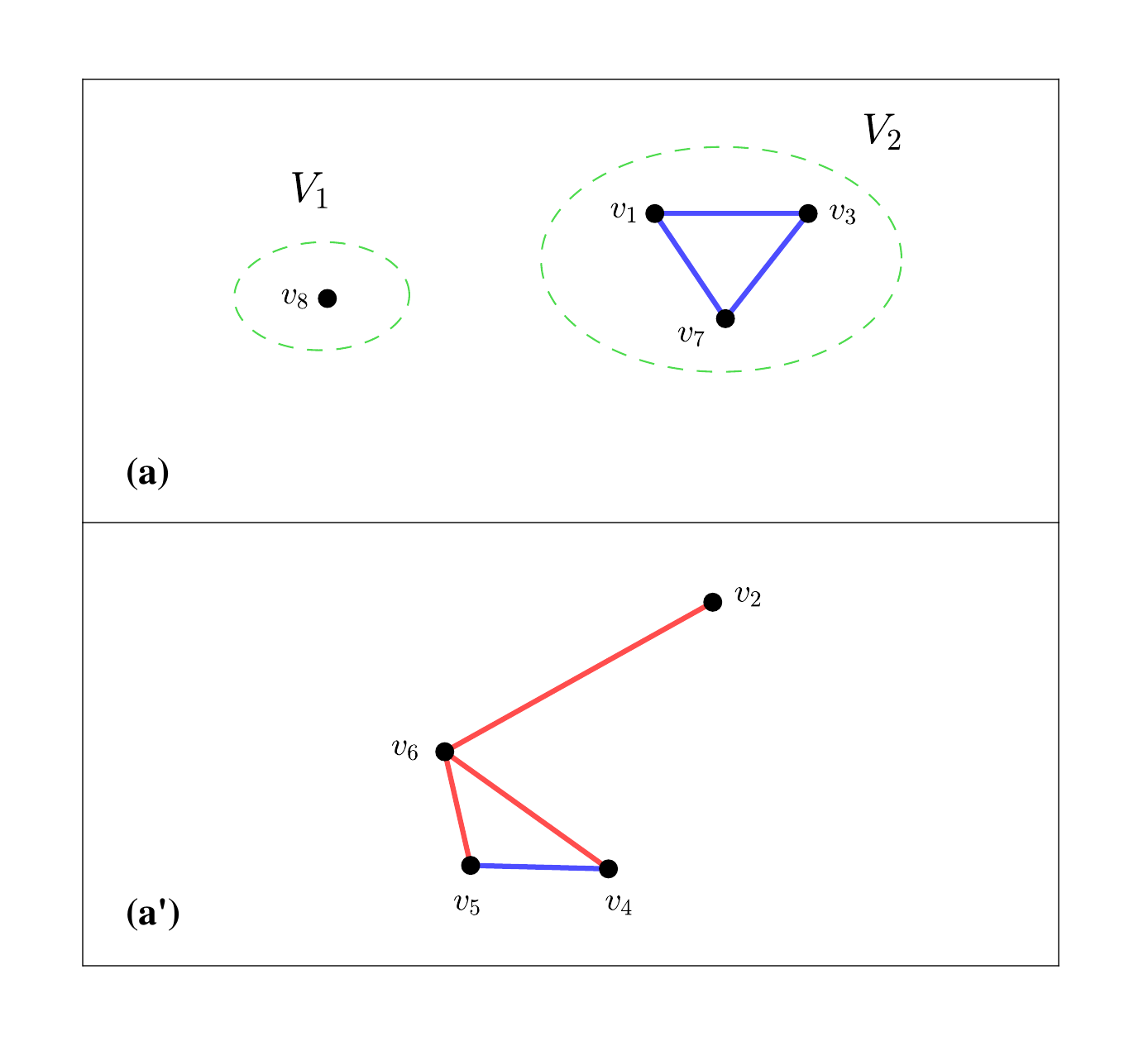}
	\caption{Steps (1) and (2) of C.C Algorithm for $\g$.}\label{ccalgo1}
	\endminipage
\end{figure}	
\begin{figure}[!htb]
	\minipage{0.62\textwidth}
	\includegraphics[width=\linewidth]{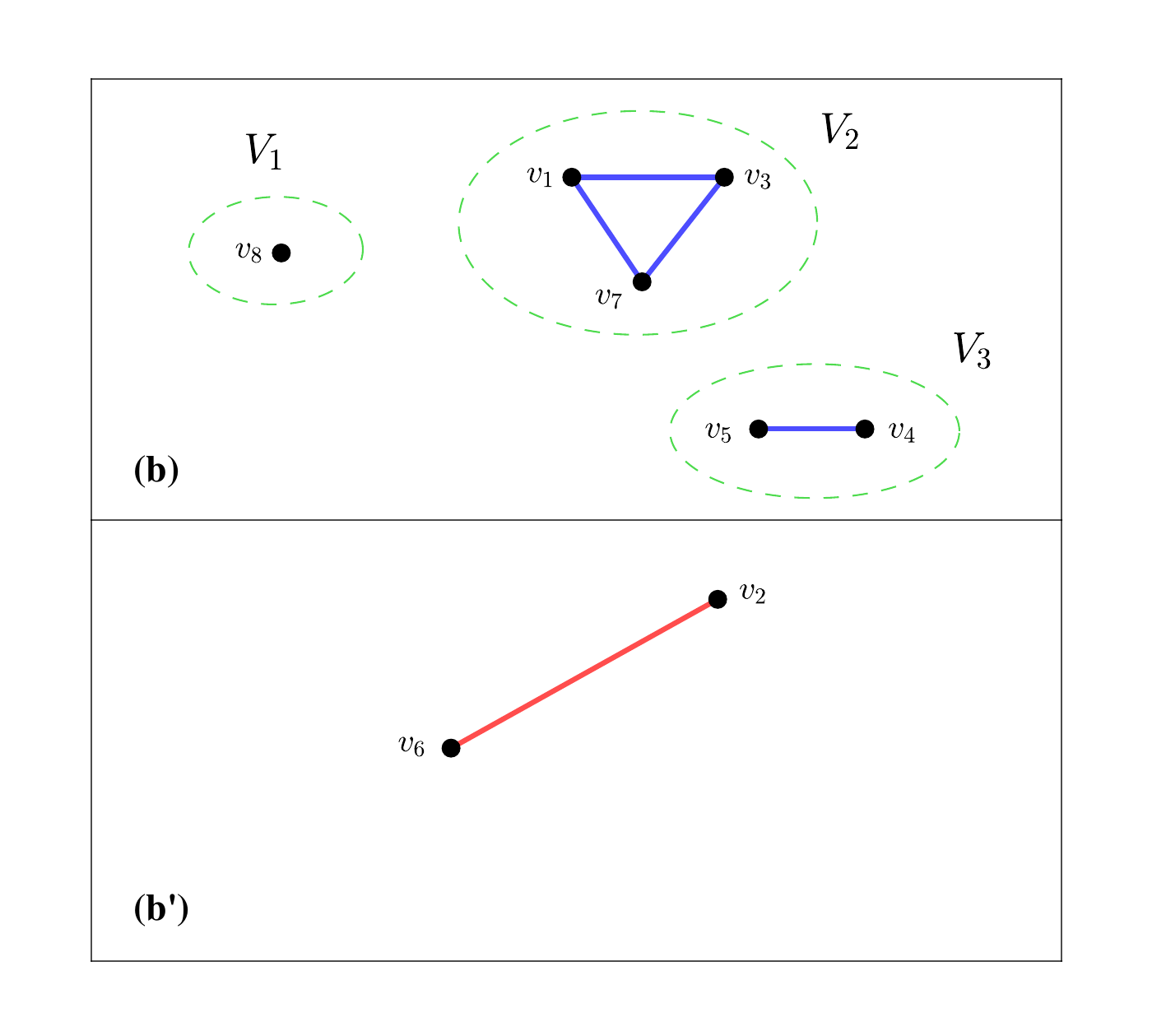}
	\caption{Steps (3) and (4) of C.C Algorithm for $\g$.}\label{ccalgo2}
	\endminipage
\end{figure}
\begin{figure}[!htb]
	\minipage{0.72\textwidth}
	\includegraphics[width=\linewidth]{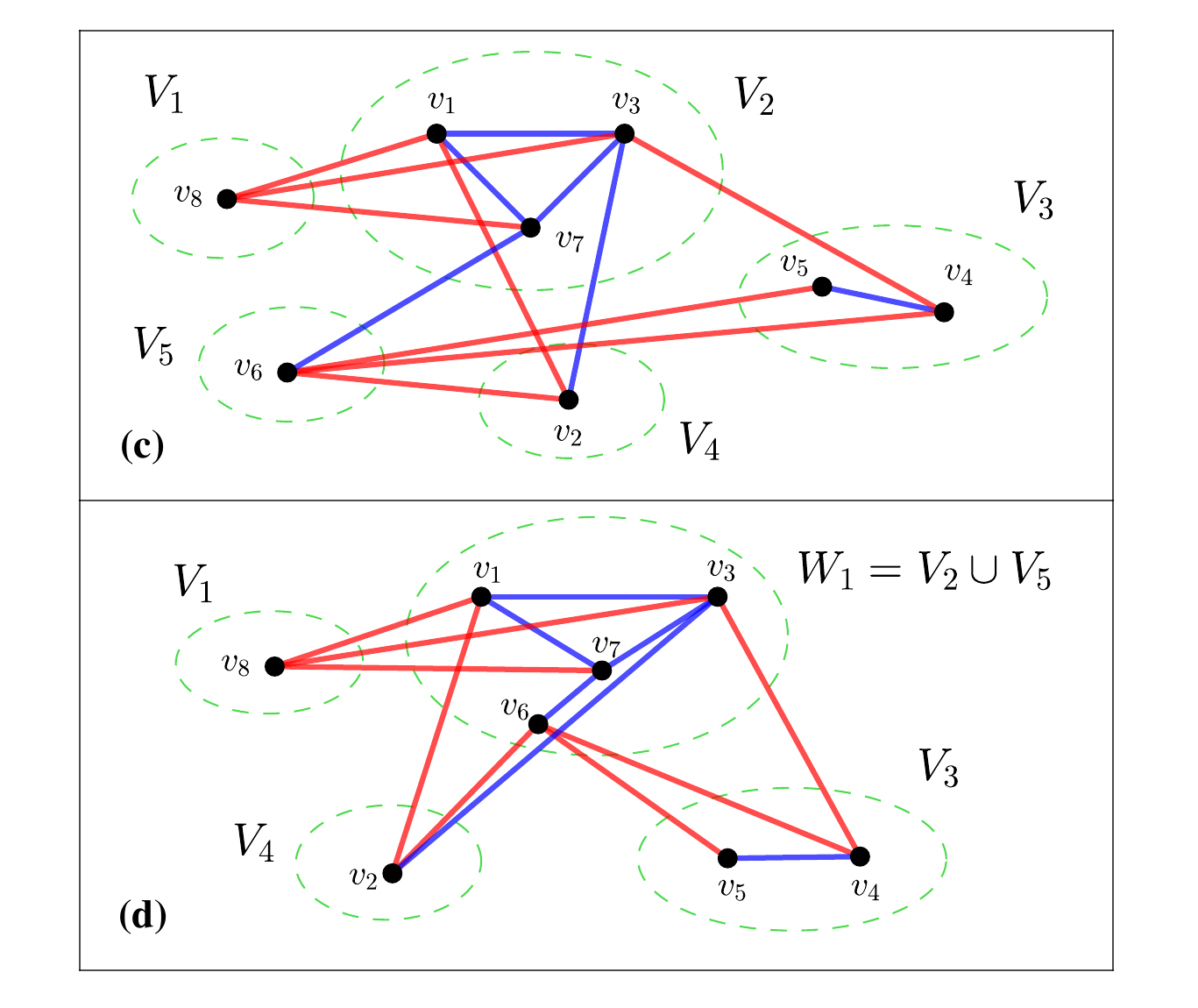}
	\caption{Correlation clustering of $\g$.}\label{ccalgo4}
	\endminipage
\end{figure}		
the graph shown in part (b$'$) in Fig. \ref{ccalgo2}) in a new cluster. So, we have a new cluster $V_4$ containing vertex $v_2$ and the other new cluster $V_5$ containing vertex $v_6$. Also, we add the edges of the initial signed graph $\g$ between the created clusters or inside them, see part (c) in Fig. \ref{ccalgo4}. By step (5), we should calculate the values $D^{+}(V_i, V_j)$ and $D^{-}(V_i, V_j)$ for $i \neq j,~ 2\leqslant i, j \leqslant 5$, for part (c) given in Fig. \ref{ccalgo4}. Note that if $D^{-}(V_i, V_j) \neq 0$, we cannot merge two clusters $V_i$ and $V_j$ together. 
	\begin{align*}	
		D^{-}(V_2, V_3) \neq 0, ~  D^{-}(V_2, V_4) \neq 0, ~ & D^{-}(V_2, V_5) = 0,\\
		& D^{+}(V_2, V_5) = 1.
	\end{align*}
	\begin{align*}	
	& D^{-}(V_3, V_4) = 0, ~  D^{-}(V_3, V_5) \neq 0, ~  D^{-}(V_4, V_5) \neq 0,\\
	& D^{+}(V_3, V_4) = 0.
\end{align*}
Based on the obtained values, we need to find clusters $V_s$ and $V_t$ with this property that $D^{+}(V_s, V_t) = \max_{i\neq j} D^{+}(V_i, V_j)$ such that $D^{-}(V_i, V_j)=0, ~ 2\leqslant i, j \leqslant 5$. Hence, it makes sense to merge clusters $V_2$ and $V_5$ together. By step (6) of C.C Algorithm, after merging clusters $V_2$ and $V_5$ into the new cluster $W_1$, we repeat the same process for the clusters $W_1, V_3$ and $V_4$ in part (d) of Fig. \ref{ccalgo4}. See the following process: 
	\begin{align*} 
		D^{-}(W_1, V_3) \neq 0, ~ D^{-}(W_1, V_4) \neq 0, ~  & D^{-}(V_3, V_4) = 0,\\
		& D^{+}(V_3, V_4) = 0.
	\end{align*}   
	Finally, according to the calculated values, merging the clusters $V_3$ and $V_4$ does not change the number of disagreements. However, if we want to reduce the number of final clusters, we can merge any pair of the clusters $V_1, V_3$ and $V_4$ (we can also merge $V_1, V_3$ and $V_4$ into one cluster) without increasing the number of disagreements. 
	By the way, part (d) in Fig. \ref{ccalgo4} shows the obtained solution by C.C Algorithm.
\end{exmp}

In the following, we are going to compare the obtained solution by C.C Algorithm with the optimal solution for a subclass of signed graphs. In fact, an optimal solution is a solution for correlation clustering that results in the minimum number of disagreements. Let $\g$ be a signed graph. In the rest of this paper, $Sol(\g)$ and $Opt(\g)$ denote the obtained solution by C.C Algorithm and the optimal solution, respectively for $\g$. Also, the number of disagreements in $Sol(\g)$ and $Opt(\g)$ are denoted by $D(Sol(\g))$ and $D(Opt(\g))$, respectively. 

By the following Lemmas, we give some more information which are needed to prove the main result, Theorem \ref{thm:finalth}. By Theorem \ref{thm:davis}, weakly negative cycles are the destructive cases to achieve a $k$-clusterable separation. 
Obviously maximum number of edge disjoint weakly negative cycles is a lower bound for minimum number of disagreements in any solution for any arbitrary signed graph. It makes sense that one thinks the disagreements in any clustering are associated with the (positive or negative) edges of weakly negative cycles. Also, we present a construction in part (iii) of Lemma \ref{lem:numneg1} showing generally the number of disagreements in the optimal solution is not equal to the maximum number of edge disjoint weakly negative cycles. Thus, $D(Opt(\g))$ might be greater than maximum number of edge disjoint weakly negative cycles. 
Moreover, we provide a necessary and sufficient condition for the equality of the lower bound, maximum number of edge disjoint weakly negative cycles, and minimum number of disagreements in optimal solution.

To separate the vertices of each of the signed graphs given in Fig. \ref{forbidg}, in the optimal solution one cannot avoid including negative edges inside a cluster. Otherwise, we cannot achieve the minimum number of disagreements. In the following, we explain how we obtain the signed graphs in Fig. \ref{forbidg} and their importance.

\begin{figure}[!htb]
	\minipage{0.90\textwidth}
	\includegraphics[width=\linewidth]{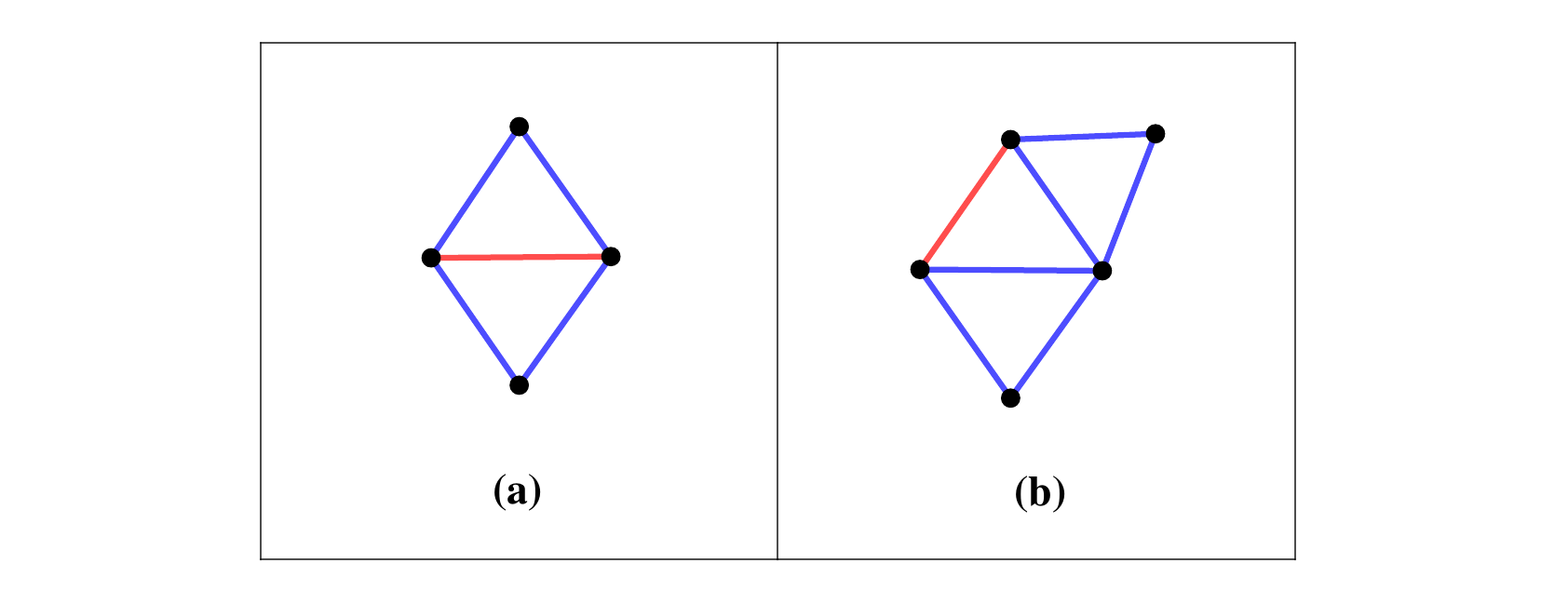}
	\caption{Two forbidden subgraphs.}\label{forbidg}
	\endminipage
\end{figure}

\begin{lem}\label{lem:forbidgraphs}
	Let $\g$ be a signed graph whose all induced strongly positive and weakly negative cycles are triangle. There exists an optimal clustering of $\g$ such that if there is a negative edge inside a cluster, it is one of the negative edges in given graphs in Fig. \ref{forbidg} as a subgraph of $\g$.  
\end{lem}
\begin{proof}
	Consider an optimal solution with the minimum number of negative edges inside the clusters. 
	Assume that there is a negative edge inside one of the clusters. The goal is to keep the minimum number of disagreements and avoid including the negative edges inside the clusters. Now, we are going to show that $\g$ contains at least one of the graphs in Fig. \ref{forbidg} as a subgraph or we can avoid including negative edges inside the clusters without increasing the number of disagreements. 
	Let $V_1$ be a cluster of the optimal solution which contains exactly one negative edge $e=v_1v_2$. If $e$ is not an edge of a cycle in $V_1$, then by adding a new cluster and separating the vertices of $e$ into two distinct clusters, we can avoid inserting the edge $e$ inside a cluster. Hence, $e$ must be an edge of a cycle with just one negative edge $e$. By assumption, $e$ is an edge of a weakly negative triangle $v_1v_2v_3$, but this is not all. If $v_1v_2v_3$ is a weakly negative triangle inside $V_1$, and the degree of $v_1$ (or $v_2$) inside $V_1$ is $2$, then actually we can remove $e$ from $V_1$ by adding a new cluster and putting $v_1$ in it without increasing the number of disagreements. Hence, the degrees of $v_1$ and $v_2$ in $V_1$ are at least $3$. If there is a vertex $v_4$ adjacent to both $v_1$ and $v_2$, we have graph (a) in Fig. \ref{forbidg} as a subgraph inside $V_1$. Otherwise, there are two vertices $v_4$ and $v_5$ inside $V_1$ distinct from $v_1$, $v_2$ and $v_3$ such that $v_4$ is adjacent to $v_1$ and $v_5$ is adjacent to $v_2$. By the former case, $v_4$ and $v_5$ are not allowed to be adjacent to $v_2$ and $v_1$, respectively. Also, if both $v_4$ and $v_5$ are not adjacent to $v_3$, we can separate these vertices into two clusters such that there is no negative edge inside a cluster. Without loss of generality, let $v_3$ and $v_4$ be adjacent, but $v_3$ and $v_5$ are not adjacent. Add a new cluster $V_2$ with $V_1=\{v_1, v_3, v_4\}$ and $V_2=\{v_2, v_5\}$.
	Moreover, if there is a path in $V_1$ from $v_4$ to $v_2$ or $v_3$ avoiding $v_1$, since all induced weakly negative and strongly positive cycles are triangle by taking chords, we also obtain the similar result. 
	Thus, $v_4$ and $v_5$ should be adjacent to $v_3$, and in this case one can find the graph (b) in Fig. \ref{forbidg} as a subgraph of $V_1$. 
	
	Now, if there exist two or more negative edges in a cluster, by Theorem \ref{thm:davis} each of the negative edges is an edge of a weakly negative cycle in the cluster. Furthermore, it is not possible to have a weakly negative cycle of order more than $3$ inside a cluster such that it contains more than one negative edges (as its edge and part of its chords) and includes only one weakly negative cycle. 
	Therefore, there are at least two weakly negative cycles in the cluster, and we have three options for weakly negative cycles: (i) they are adjacent with a common edge, (ii) they do not have a common edge, but each pair of them is adjacent, or (iii) they are edge disjoint weakly negative cycles. For (i), there is a separation of these vertices with only one disagreement (the common edge) associated with two adjacent weakly negative cycles or more than two weakly negative cycles with a common edge, and putting negative edges inside the cluster does not provide the optimal solution. For (ii), the number of disagreements associated with such case is at least $2$, but we can add new clusters and avoid inserting negative edges inside the clusters without increasing the number of disagreements. For (iii), we can repeat the proof of existence one negative edge inside a cluster for each of these negative edges. Besides, if we have a mix case containing at least two cases of these three mentioned cases, one by one we should consider them based on above explanations. 
	Note that by Theorem \ref{thm:davis}, we ignore induced cycles with two or more negative edges in a cluster. 
	
\end{proof}

\begin{rem}
	In Lemma \ref{lem:forbidgraphs}, note that if one of the positive edges of graph (a) belongs to another weakly negative cycle and is not over a strongly positive triangle, then we can avoid inserting negative edge associated with graph (a) inside the clusters. Also, if we add a negative edge to graph (b) we can also have a separation of vertices such that there is no negative edge corresponding to subgraph (b) inside the clusters. 
\end{rem}

\begin{figure}[!htb]
	\minipage{0.90\textwidth}
	\includegraphics[width=\linewidth]{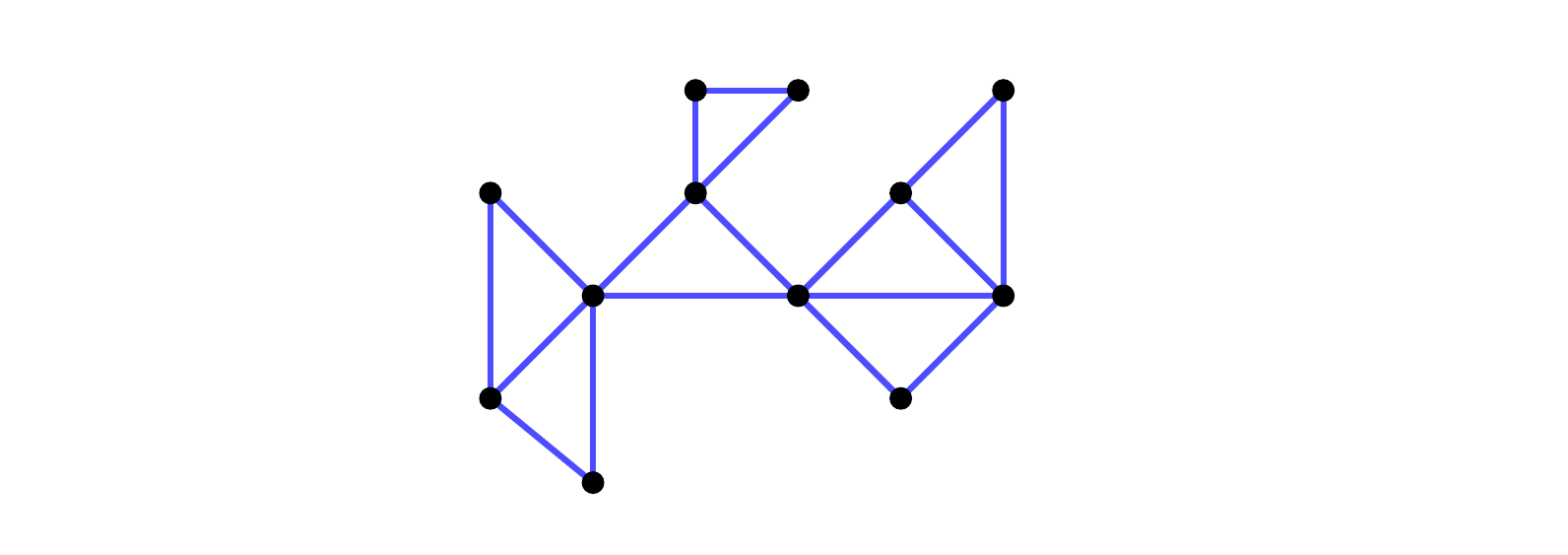}
	\caption{A chain of strongly positive triangles.}\label{tri_chain}
	\endminipage
\end{figure}

\begin{lem}\label{lem:avoidneg}
	Let $\g$ be a signed graph whose all induced strongly positive and weakly negative cycles are triangle.  
	Consider a subgraph of $\g$ containing a chain of strongly positive triangles which are connected to each other by a common edge or a vertex, and there are no other edges (negative or positive) in the subgraph, like given graph in Fig. \ref{tri_chain}. By adding any negative edge to this chain, one of the subgraphs in Fig. \ref{forbidg} is made.
\end{lem}

\begin{proof}
	Consider a chain with maximum number of strongly positive triangles, and add a negative edge to this chain. Since all induced weakly negative cycles are triangle, the negative edge is over a weakly negative triangle. Now, we prove the assertion by induction on $p$, the number of strongly positive triangles in the chain. If $p=2$, there is a chain with two strongly positive triangles. Then we have two cases for strongly positive triangles, (i) they are common in an edge, (ii) they are common in a vertex. In both cases by adding a negative edge, obviously we can find one of the graphs in Fig. \ref{forbidg} as a subgraph. Assume that the assertion is true for any chain containing $p$ strongly positive triangles for $p\geqslant 2$. Now, consider a chain which the maximum number of strongly positive triangles is $p+1$, and add a negative edge to this chain. This negative edge creates a weakly negative triangle with two positive edges in the chain. Because all induced weakly negative cycles in $\g$ are triangle, and also the chain contains maximum number of strongly positive triangles. Consider $v_1v_2v_3$ as the weakly negative triangle with negative edge $v_1v_2$. We know that $v_1v_3$ and $v_2v_3$ belong to two strongly positive triangles in the chain. Consider those two strongly positive triangles and the negative edge. By the first step of induction (for $p=2$) one of the graphs (a) or (b) in Fig. \ref{forbidg} is a subgraph of $\g$. 
\end{proof}

We can see the correlation clustering problem with a different view. Consider a signed graph $\g$, and let $U$ be a subset of $E$ such that if we remove the edges in $U$ from $\g$, the obtained signed graph has no weakly negative cycle, see \cite{Zaslavsky}. Now, if a set $U$ has the minimum size with this property, then we have $|U| = D(Opt(\g))$. We know that for each edge disjoint weakly negative cycle there is a disagreement in the optimal solution. Then, there is an edge in $U$ corresponding to each edge disjoint weakly negative cycle in $\g$. In Lemma \ref{lem:negdisad}, we show that if all weakly negative cycles in $\g$ are edge disjoint, that is, there are no two adjacent weakly negative cycles in $\g$, then minimum number of disagreements is equal to the number of weakly negative cycles in $\g$. But, if we have two or more weakly negative cycles with a common edge or more than two weakly negative cycles such that each pair of them is adjacent, we should consider possible cases and find th appropriate edge(s) as members of $U$. See Lemma \ref{lem:numneg1}.  

\begin{lem}\label{lem:negdisad}
	Let $\g$ be a signed graph whose all weakly negative cycles are edge disjoint. That is, there are no two adjacent weakly negative cycles in $\g$. Then, minimum number of disagreements is equal to the number of weakly negative cycles in $\g$.	
\end{lem}
\begin{proof}
	We know that maximum number of edge disjoint weakly negative cycles is a lower bound for minimum number of disagreements. Now, consider weakly negative cycles one by one and delete an edge (negative edge) of each of them such that it causes the corresponding weakly negative cycle to disappear. If there are two adjacent weakly negative cycles, only in this case, by removing an edge of one of the weakly negative cycles, another weakly negative cycle may remain. So, by removing negative edges of weakly negative cycles, the obtained signed graph from $\g$ does not contain any weakly negative cycle. So, by Theorem \ref{thm:davis} the obtained graph is $k$-clusterable for some $k$, and there is a separation of vertices with no disagreement. Then add the removed negative edges to this clustering. Hence, the number of weakly negative cycles is an upper bound for minimum number of disagreements in $\g$. We conclude that minimum number of disagreements of $\g$ is equal to the number of weakly negative cycles in $\g$. 
\end{proof}

\begin{lem}\label{lem:numneg1}
	Let $\g$ be a signed graph. 
	\begin{itemize}
		\item[(i).] Let $H$ be an induced subgraph of $\g$. If minimum number of disagreements of $H$ is $k$, then there are $k$ disagreements in the optimal clustering of $\g$ associated with $H$. 
		\item[(ii).] Let $\g$ contain two adjacent weakly negative cycles such that $\g$ has no other weakly negative cycle. Then, $D(Opt(\g)) =1$. Moreover, if $\g$ contains a number of weakly negative cycles such that all of them are in a sequence and only two consecutive ones are adjacent, minimum number of disagreements is equal to maximum number of edge disjoint weakly negative cycles. And, $\g$ does not contain any other weakly negative cycle. 
		\item[(iii).] Let $\g$ have just one edge disjoint weakly negative cycle but contain three weakly negative cycles, so there are two options: (a) these three weakly negative cycles are common in at least one edge, (b) these three weakly negative cycles do not have a common edge, but each pair of them is adjacent. For case (a), $D(Opt(\g)) = 1$, but for case (b), $D(Opt(\g)) = 2$.
		
	\end{itemize}
\end{lem}
\begin{proof}
	\begin{itemize}
		\item[(i).]
		Consider an optimal clustering for $H$, and apply the same separation to vertices of $H$ in $\g$. Note that there is no separation of vertices of $H$ in $\g$ such that minimum number of disagreements associated with $H$ is less than $k$. If there is such separation, we can use it for the signed graph $H$, and it is in conflict with the assumption. 
		
		\item[(ii).] 
		Consider two adjacent weakly negative cycles with at least one common edge (positive or negative). It is easy to check that if we remove one of the common edge, $\g$ does not contain any weakly negative cycle. Thus, $D(Opt(\g))=1$. Besides, if $\g$ contains a sequence of such adjacent weakly negative cycles, it is sufficient to remove the minimum number of common edges of the adjacent weakly negative cycles which is equal to maximum number of edge disjoint weakly negative cycles so that the weakly negative cycles disappear.
		
		\item[(iii).] Assume that we have three weakly negative cycles (with just one edge disjoint weakly negative cycle). If these three weakly negative cycles have at least one common edge, $D(Opt(\g))=1$, and by deleting one of the common edges all three weakly negative cycles disappear. Now, if these three weakly negative cycles do not have a common edge, since there is just one edge disjoint weakly negative cycle in $\g$, each pair of the three weakly negative cycles is adjacent. One can check that by removing any edge (one edge) at least one weakly negative cycle remains, so $D(Opt(\g)) > 1$. But, by removing two suitable edges among common edges, the obtained signed graph from $\g$ does not include any weakly negative cycle. Hence, $D(Opt(\g))=2$. 
	\end{itemize}
\end{proof}

\begin{figure}[!htb]
	\minipage{0.95\textwidth}
	\includegraphics[width=\linewidth]{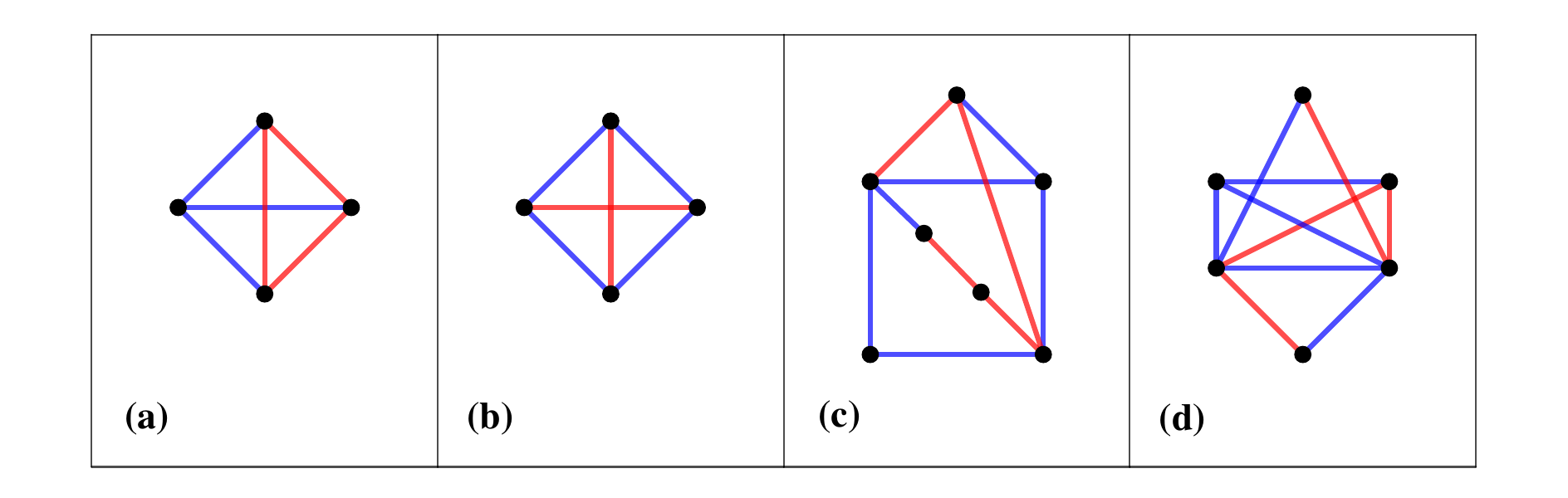}
	\caption{Signed graphs whose minimum number of disagreements is larger than the number of edge disjoint weakly negative cycles.}\label{example}
	\endminipage
\end{figure}

According to part (iii) of Lemma \ref{lem:numneg1}, we can find many examples that minimum number of disagreements is not equal to the maximum number of edge disjoint weakly negative cycles. As some examples, see graphs given in Fig. \ref{example}. Each of the graphs (a), (b) and (c) contains exactly one edge disjoint weakly negative cycle, but minimum number of disagreements is $2$. Graph (d) includes two edge disjoint weakly negative cycles but minimum number of disagreements is $3$. 

\begin{thm}\label{thm:condition}
	Let $\g$ be a signed graph. $\g$ does not contain any three weakly negative cycles without a common edge such that each pair of them is adjacent (that is, if there are three weakly negative cycles in $\g$ such that each pair of them is adjacent, then those three weakly negative cycles have a common edge), if and only if minimum number of disagreements is equal to maximum number of edge disjoint weakly negative cycles in $\g$.
\end{thm}

\begin{proof}
	Assume that all three weakly negative cycles in $\g$, with this property that each pair of them is adjacent, have a common edge. Consider an induced subgraph $F$ containing just one edge disjoint weakly negative cycle which may be adjacent to a number of weakly negative cycles. It is obvious that each pair of these weakly negative cycles is also adjacent, so according to the assumption, all the weakly negative cycles in $F$ have a common edge. Removing the common edge makes all weakly negative cycles in $F$ disappear. Apply this method to other induced subgraphs of $\g$ containing just one edge disjoint weakly negative cycle to remove all the weakly negative cycles in $\g$. For the converse, let minimum number of disagreements be equal to maximum number of edge disjoint weakly negative cycles in $\g$. Assume that $\g$ has a subgraph $H$ containing three weakly negative cycles without a common edge such that each pair of weakly negative cycles is adjacent, and $H$ does not contain any other weakly negative cycle. By part (iii) of Lemma \ref{lem:numneg1} the minimum number of disagreements of $H$ is two, and also by part (i) of Lemma \ref{lem:numneg1}, there are two disagreements associated with $H$ in $\g$ while it has only one edge disjoint weakly negative cycle. Obviously, this is a contradiction. Hence, if there are three weakly negative cycles in $\g$ such that each pair of them is adjacent, then those three weakly negative cycles have a common edge.
\end{proof}

We need the following remark to prove the results presented in Theorem \ref{thm:finalth}.

\begin{rem}\label{rem:rth}
	Let $\g$ be a signed graph whose all induced strongly positive and weakly negative cycles are triangle. Let $e$ be a disagreement of $Sol(\g)$. We have the following notes: 
	\begin{itemize}
		\item Although each induced weakly negative cycle in $\g$ is triangle, it does not mean that every edge disjoint weakly negative cycle in $\g$ is a triangle. As a simple example, the signed graph given in Fig. \ref{negtri} has two edge disjoint weakly negative cycles of orders $3$ and $4$, and the number of its disagreements is $2$. Note that this signed graph has just one edge disjoint weakly negative triangle.  
		
	\end{itemize}
	Now, let $e$ be a disagreement of $Opt(\g)$.
	\begin{itemize}
		\item If $e$ is an edge over a weakly negative cycle (not a triangle), then there are two options: (i) $e$ is over a strongly positive triangle, as an example see Fig. \ref{negtri}, (ii) $e$ is an edge of an induced cycle whose containing at least two negative edges. For (i), $e$ is not a disagreement in the obtained solution since by C.C Algorithm we put the vertices of any strongly positive triangle in a cluster. For (ii), corresponding to the disagreement $e$ in $Opt(\g)$, there is a disagreement in $Sol(\g)$, but it might not be $e$. Because in this situation C.C Algorithm picks a maximum matching, and then suitable clusters are merged. 	
	\end{itemize}
\end{rem}

\begin{figure}[!htb]
	\minipage{0.80\textwidth}
	\includegraphics[width=\linewidth]{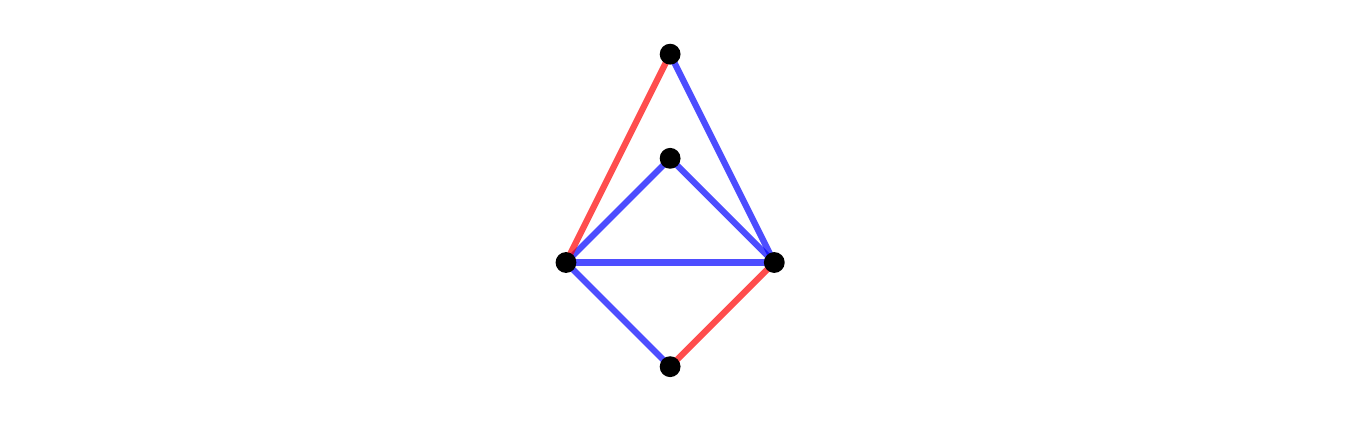}
	\caption{A signed graph with two disagreements.}\label{negtri}
	\endminipage
\end{figure}
\begin{figure}[!htb]
	\minipage{0.90\textwidth}
	\includegraphics[width=\linewidth]{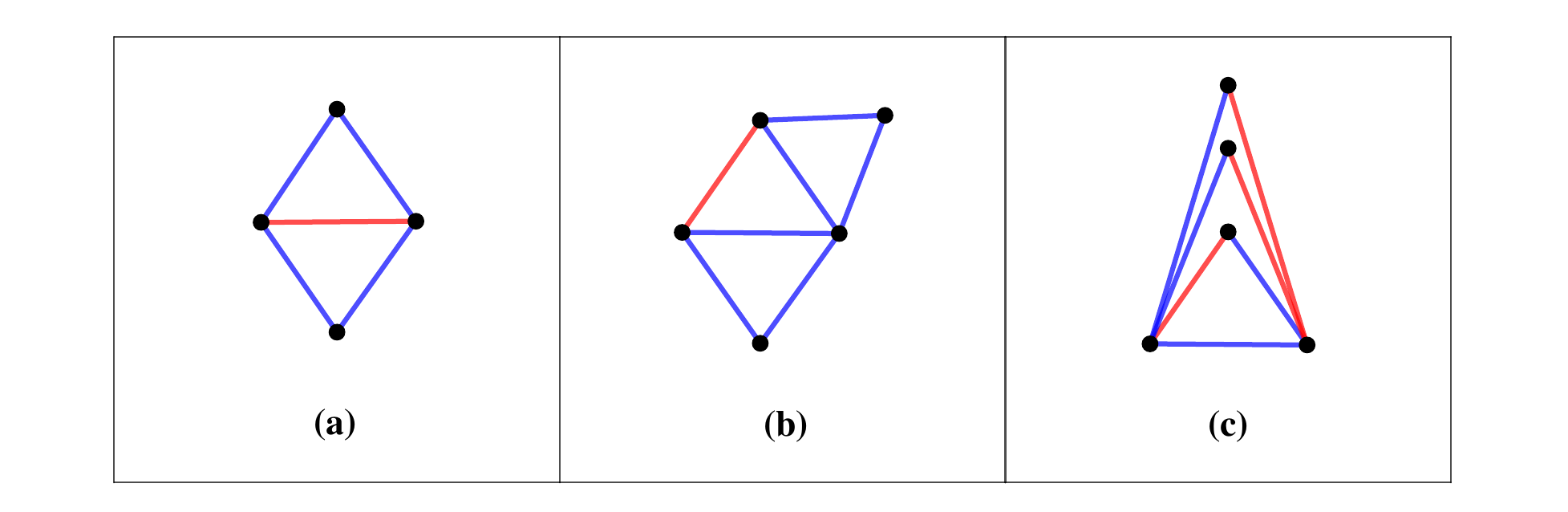}
	\caption{Three forbidden subgraphs. Graph (c) contains any three weakly negative triangles with a common positive edge.}\label{thforbid}
	\endminipage
\end{figure}
\begin{thm}\label{thm:finalth}
	Let $\g$ be a signed graph whose all induced strongly positive and weakly negative cycles are triangle. Consider graphs given in Fig. \ref{thforbid} as forbidden subgraphs for $\g$.
	Then C.C Algorithm gives a $2$-approximation for correlation clustering of $\g$.
\end{thm}
\begin{proof}
	Let $e$ be a disagreement of $Sol(\g)$. In this step, we present all types of disagreements of $Opt(\g)$ corresponding to a disagreement of $Sol(\g)$. See the following notes: 
	\begin{itemize}
		\item According to the steps of C.C Algorithm and Lemma \ref{lem:avoidneg}, we can avoid inserting negative edges inside the clusters. Hence, $e$ is not a negative edge inside a cluster of $Sol(\g)$, so all disagreements in $Sol(\g)$ are positive.
		
		\item Positive edges of induced cycles with at least two negative edges (which are not over weakly negative cycles) are not disagreements of $Sol(\g)$ because choosing the maximum matching of positive edges and placing each of them in a cluster prevents this from happening. If a positive edge is not a member of the maximum matching and is not an edge of a strongly positive triangle or a weakly negative triangle, in step (4) it adds to one of the clusters. 
		
		\item If $e=xy$ is an edge of a weakly negative triangle, then we have three options for the optimal solution listed in Fig. \ref{options}. Note that by Lemma \ref{lem:forbidgraphs} there is no negative edge inside the clusters in optimal solution.
			
		\item If $e$ is not an edge of a weakly negative triangle, then it is an edge of a weakly negative cycle of order more than $3$, as well as it is an edge of an induced cycle with at least two negative edges, see Remark \ref{rem:rth}.
	\end{itemize}
\begin{figure}[!htb]
	\minipage{0.90\textwidth}
	\includegraphics[width=\linewidth]{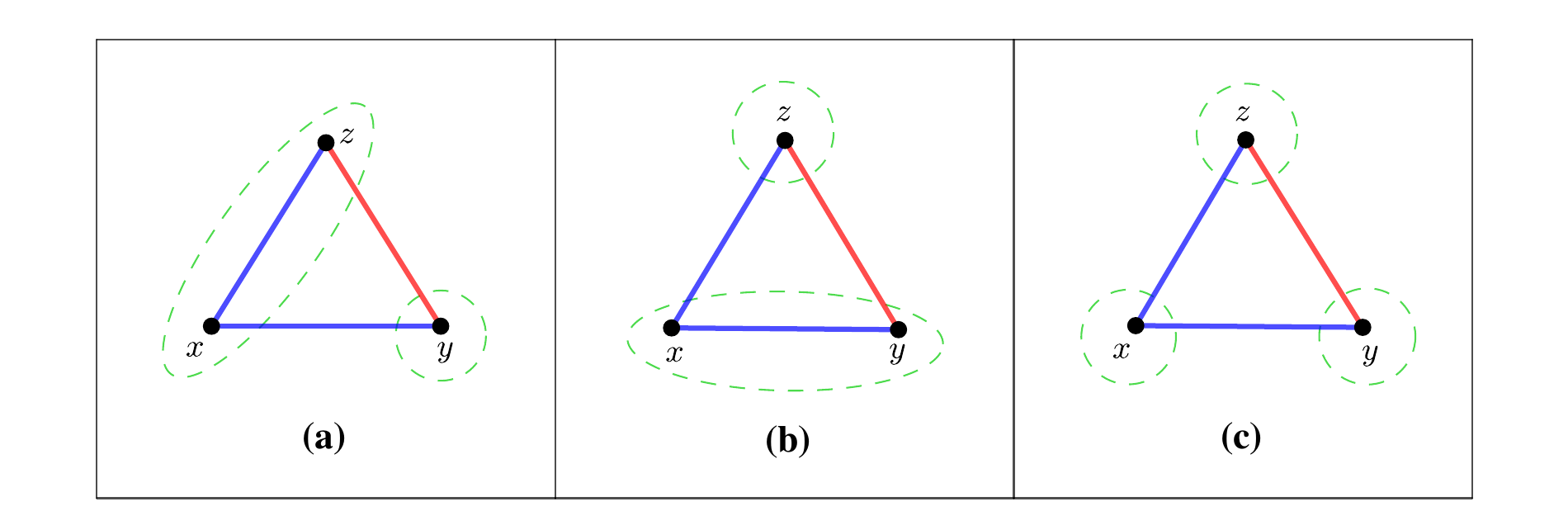}
	\caption{A disagreement in a cluster.}\label{options}
	\endminipage
\end{figure}	
	In this step, we verify the existence of disagreements of $Sol(\g)$ based on a disagreement in $Opt(\g)$. Actually, this step is the converse of the former step. Now, let $e=xy$ be a disagreement of $Opt(\g)$. By Lemma \ref{lem:forbidgraphs} and the forbidden subgraphs (a) and (b) in Fig. \ref{thforbid}, $e$ cannot be a negative edge inside a cluster. So, the disagreement $e$ is a positive edge between two clusters. \\
	Case (1): if $e$ is an edge of exactly one weakly negative triangle, then there are at most two disagreements in $Sol(\g)$ corresponding to $e$, see Fig. \ref{options}. 
	
	\begin{figure}[!htb]
		\minipage{0.90\textwidth}
		\includegraphics[width=\linewidth]{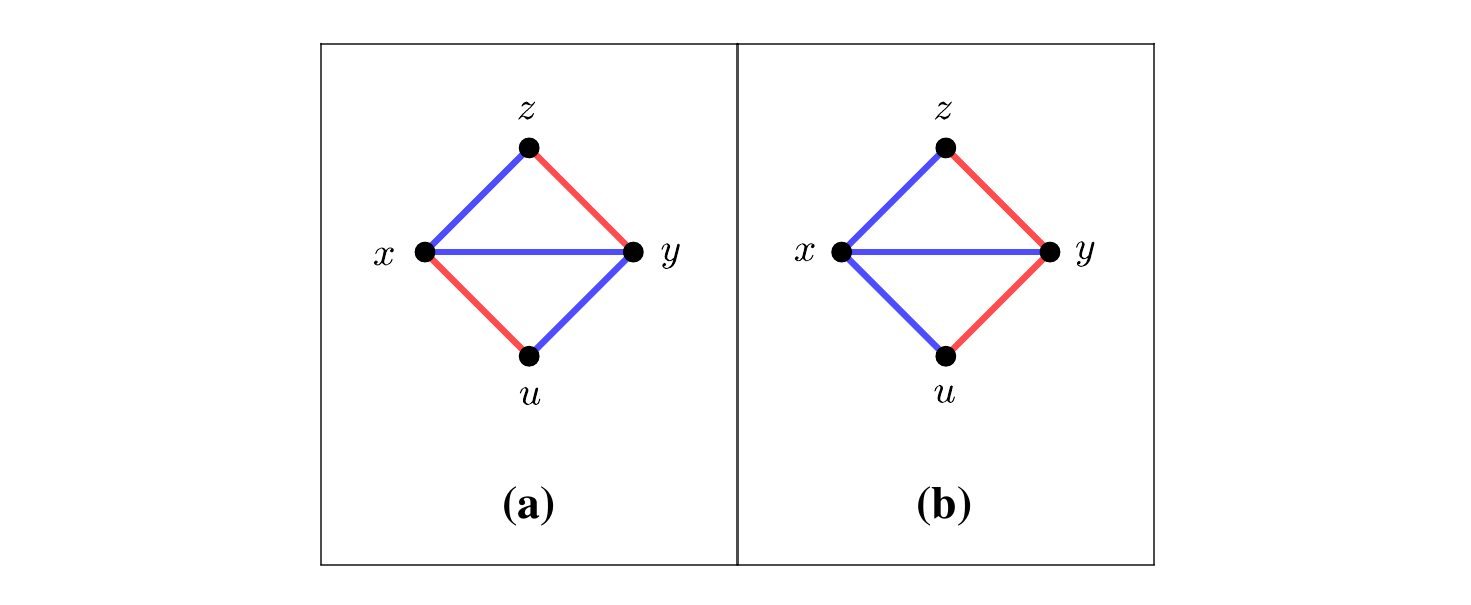}
		\caption{Two weakly negative triangles with a common positive edge.}\label{twotri}
		\endminipage
	\end{figure}

	Case (2): if $e$ is a common edge between two weakly negative triangles, see Fig. \ref{twotri}. We should compute the number of disagreements in $Sol(\g)$ associated with $e$. For graph (a), let the edges $xy, xz$ and $yu$ be three disagreements in $Sol(\g)$. Then at least two vertices of $x, y, z, u$ must be over two edges of the matching considered in step (3) of C.C Algorithm. If they are $x$ and $u$, denote them by $e_1=xw$ and $e_2=up$, for some $w,~p\in V$. Eventually, we have two options here (i) without loss of generality, suppose that $wz$ is a negative edge, then we are not allowed to merge the vertex $z$ in the cluster containing $e_1$. So, the number of edge disjoint weakly negative triangles is two and we have three disagreements. And, (ii) if we can merge the vertex $z$ in the cluster containing $e_1$. The number of disagreements is at most two, so we achieve a contradiction because we suppose that the edges $xy, xz$ and $yu$ are three disagreements in $Sol(\g)$. Moreover, we have a similar condition for $e_2$ and $y$. Hence, considering all cases for subgraph (a) in Fig. \ref{twotri} implies that just in one case the number of disagreements is $3$, when there exists another weakly negative triangle containing $xz$ or $yu$. Now consider graph (b) in Fig. \ref{twotri}. If at least one of the edges $xz$, $xy$ or $xu$ is an element of the maximum matching, then the number of disagreements is at most $2$. If non of these edges is an element of the maximum matching, and $e_3=xf$ for some $f\in V$ is an element of maximum matching, then we can merge the vertex $z$ in the cluster containing $e_3$, unless without loss of generality, $fz$ is a negative edge. In this condition, the number of edge disjoint weakly negative triangles is two and the number of disagreements is $3$. Now, let $yq$ for some $q\in V$ be an element of maximum matching. By the forbidden subgraph (c) in Fig. \ref{thforbid}, we can merge $x$ in the cluster containing $yq$, and the number of disagreements is two. \\
	Case (3): if $e$ is an edge of a weakly negative cycle $C_1$ of order more than $3$ such that $e$ is also an edge of an induced cycle with at least two negative edges. If none of the positive edges of $C_1$ and its chords is an edge of other weakly negative cycles, then for $Sol(\g)$ also there is only one disagreement associated with this weakly negative cycle. But, if at least one of the positive edges is an edge of another weakly negative cycle, the number of disagreements over $C_1$ or its chords may increase based on the number of other weakly negative cycles adjacent with $C_1$.  
	Finally, we can conclude that $D(Sol(\g)) \leqslant 2D(Opt(\g))$. 
\end{proof}

\section{Conclusion}

	 In this paper, we presented an algorithm for correlation clustering in general case (C.C Algorithm on page 4). Also, we showed that there is a necessary and sufficient condition under which the lower bound, maximum number of edge disjoint weakly negative cycles, is equal to minimum number of disagreements, see Theorem \ref{thm:condition}. Moreover, we proved that the C.C Algorithm provides a $2$-approximation for a subclass of signed graphs, see Theorem \ref{thm:finalth}. 
	 
	 As a result of Theorem \ref{thm:finalth}, if $\g$ is a signed chordal graph and graphs given in Fig. \ref{thforbid} are forbidden subgraphs for $\g$, then C.C Algorithm gives a $2$-approximation for correlation clustering of $\g$. Now, why are chordal graphs important? 
	 Chordal (also known as a triangulated) graphs are one of the most extensively studied classes of graphs. Chordal graphs arise in many practical and relevant fields such as computing the solutions of systems of linear equations, in database management systems, VLSI, biology, and so on. We refer the reader to \cite{chorsur} for further applications of chordal graphs. 
	 
	 In addition, the complexity of the C.C Algorithm considering the condition given in Theorem \ref{thm:finalth} is $\mathcal{O}(n^3)$. Actually, the complexity for detecting all chains of strongly positive triangles is $\mathcal{O}(n^3)$, and the complexity of finding a maximum matching among the positive edges is also $\mathcal{O}(n^3)$. 
	Besides for signed graphs without weakly negative cycles, the C.C Algorithm results in a clustering without any disagreements. Because there are no negative edges inside the clusters and between each pair of clusters there are only positive edges or only negative edges, not both. Otherwise, it creates at least one weakly negative cycle. Hence, we can merge clusters that only have positive edges between them. Finally, we separate the vertices into a number of clusters without any disagreements.
	
	As we know, in 2006, Demaine et al. in \cite{cc06} have shown that the problem for general graphs is equivalent to minimum multicut problem, so it is APX-hard and difficult to approximate better than $\Theta(\mathrm{log}n)$. They also gave an $\mathcal{O}(\mathrm{log}n)$-approximation algorithm for general graphs. Moreover, they provided an $\mathcal{O}(r^3)$-approximation algorithm for $K_{r,r}$-minor-free graphs. In 2021, Mandaglio et al. proved for general graphs satisfying Global Weight Bound (GWB) condition, there is a $5$-approximation algorithm for correlation clustering. Of course, signed graphs (not signed complete graphs) do not satisfy the GWB condition unless we change the edge weights.\\
	Our main goal is to improve results in correlation clustering for signed general graphs, so for the next step, it is worthwhile to work on a larger subclass or other subclasses of signed general graphs such as signed chordal graphs without forbidden subgraphs.


\begin{thebibliography}{10}
	
	\bibitem{cc}{\sc N. Bansal}, {\sc A. Blum}, and {\sc S. Chawla}, {\em Correlation clustering}, Mach. Learn. 56 (2004), no. 1-3, 89--113.
	
	\bibitem{ccbook22}{\sc F. Bonchi}, {\sc D. Garc\'{i}a-Soriano}, {\sc F. Gullo}, , {\em Correlation Clustering}, 2022, Data Mining and Knowledge Discovery, Statistics, general, Springer. https://doi.org/10.1007/978-3-031-79210-6
	
	\bibitem{cc15}{\sc S. Chawla}, {\sc K. Makarychev}, {\sc T. Schramm}, and {\sc G. Yaroslavtsev}, {\em Near optimal LP rounding algorithm for correlation clustering on complete and complete $k$-partite graphs}, STOC'15—Proceedings of the 2015 ACM Symposium on Theory of Computing, 219--228, ACM, New York, 2015.
	
	\bibitem{cc24}{\sc V. Cohen-Addad}, {\sc D. R. Lolck}, {\sc M. Pilipczuk}, {\sc M. Thorup}, {\sc S. Yan}, {\sc H. Zhang}, {\em Combinatorial Correlation Clustering}, STOC 2024: Proceedings of the 56th Annual ACM Symposium on Theory of Computing
	Pages 1617--1628. https://doi.org/10.1145/3618260.3649712
	
	\bibitem{Davis}{\sc J.A. Davis}, {\em Clustering and structural balance in graphs}, Human Relations 20
	(1967), 181--187. Reprinted in: Samuel Leinhardt, ed., Social Networks: A Developing
	Paradigm, pp. 27--33. Academic Press, New York, 1977.
	
	\bibitem{cc06}{\sc E.D. Demaine}, {\sc D. Emanuel}, {\sc A. Fiat}, and {\sc N. Immorlica}, {\em Correlation clustering in general weighted graphs}, Theoret. Comput. Sci. 361 (2006), no. 2-3, 172--187.
	
	\bibitem{[16]}{\sc P. N. Klein}, {\sc S. A. Plotkin}, {\sc S. Rao}, {\em Excluded minors, network decomposition, and multicommodity flow}, In Proceedings of the 25th Annual ACM Symposium on Theory of Computing, pages 682--690, 1993.
	
	\bibitem{[17]}{\sc T. Leighton}, {\sc S. Rao}, {\em Multicommodity max-flow min-cut theorems and their use in
	designing approximation algorithms}, Journal of the ACM, 46(6):787--832, 1999.

    \bibitem{ccbound}{\sc D. Mandaglio}, {\sc A. Tagarelli}, {\sc F. Gullo}, {\em Correlation Clustering with Global Weight Bounds}, in: Proc. ECML-PKDD, 2021, pp. 499--515. 
	
	\bibitem{chorsur}{\sc Md. Z. Rahman}, {\em Chordal Graphs and Their Relatives: Algorithms and Applications}, University of Windsor (Canada) ProQuest Dissertations and Theses, 2020.
	
    \bibitem{[27]}{\sc E. Tardos}, {\sc V. V. Vazirani}, {\em Improved bounds for the max-flow min-multicut ratio for planar and $K_{r,r}$-free graphs}, Information Processing Letters, 47(2):77--80, 1993.

	\bibitem{Zaslavsky}{\sc T. Zaslavsky}, {\em Balance and clustering in signed graphs}, Slides from lectures at the C.R. Rao Advanced Institute of Mathematics, Statistics and Computer Science, Univ. of Hyderabad, India, 2010.
	
\end{thebibliography}
\end{document}